\documentclass[11pt,twoside]{article}

\def\titlename{\huge Amenability notions of hypergroups  and some applications to locally  compact groups}

\title{\titlename}

\def\authname{Mahmood Alaghmandan}
\author{{\normalsize\sc \authname}}

\usepackage{amsmath}

\usepackage{yfonts}

\usepackage{xcolor}
\definecolor{blue1}{RGB}{32,78,170}
\definecolor{blue2}{RGB}{93,92,160}
\definecolor{blue3}{RGB}{40,51,202}
\definecolor{blue4}{RGB}{0,0,0}
\definecolor{purple1}{RGB}{128,0,128}

\usepackage[all]{xy}

\definecolor{El}{rgb}{.4,.9,1}
\usepackage[pdftex,
            pdfauthor={Mamhood},
            linkcolor=blue1,
            citecolor=blue3,
            colorlinks=true]{hyperref}

\usepackage{titlesec}

\titleformat{\section}
  {\normalfont\fontsize{12}{15}\bfseries}{\thesection}{1em}{}

\usepackage[english]{babel}
\usepackage{blindtext}
\usepackage[scaled=.90]{helvet}
\usepackage{xcolor}
\usepackage{titlesec}
\titleformat{\chapter}[display]
  {\normalfont\sffamily\huge\bfseries\color{blue4}}
  {\chaptertitlename\ \thechapter}{20pt}{\Huge}
\titleformat{\section}
  {\normalfont\sffamily\Large\color{blue4}}
{\thesection}{1em}{}

 \titleformat{\subsection}
  {\normalfont\sffamily\large\color{blue4}}
  {\thesubsection}{1em}{}

\newcommand{\ma}[1]{\textcolor{blue2}{\textsf{#1}}}

            \newcounter{pulse}[section]
\numberwithin{pulse}{section}
\numberwithin{equation}{section}

\newtheorem{theorem}[pulse]{\bf \textsf{Theorem}}
\newtheorem{proposition}[pulse]{\bf \textsf{Proposition}}
\newtheorem{lemma}[pulse]{\bf \textsf{Lemma}}

\newtheorem{prop}[pulse]{\bf \textsf{Proposition}}
\newtheorem{lem}[pulse]{\bf \textsf{Lemma}}
\newtheorem{cor}[pulse]{\bf \textsf{Corollary}}

\newtheorem{dummy-eg}[pulse]{\bf \textsf{Example}}
\newtheorem{dummy-rem}[pulse]{\bf \textsf{Remark}}
\newenvironment{eg}{\begin{dummy-eg}\upshape}{\end{dummy-eg}\ignorespacesafterend}
\newenvironment{rem}{\begin{dummy-rem}\upshape}{\end{dummy-rem}\ignorespacesafterend}

\newtheorem{dummy-def}[pulse]{\bf \textsf{Definition}}

\usepackage{amsmath,amssymb}

\newenvironment{dfn}{\begin{dummy-def}\upshape}{\end{dummy-def}\ignorespacesafterend}

\newenvironment{proof}{\noindent{\it Proof.}\/}{\hfill $\Box$ \newline \ignorespacesafterend}

\newcommand{\supp}{\operatorname{supp}}

\newcommand{\conj}{\operatorname{Conj}}

\renewcommand{\ker}{\operatorname{Ker}}

\newcommand{\wH}{{\widehat{H}}}
\newcommand{\wG}{{\widehat{G}}}

\newcommand{\Zlg}{{Z\ell^1(G)}}

\newcommand{\Nat}{{\mathbb N}}

\newcommand{\norm}[1]{\Vert #1 \Vert}

\newcommand{\Ind}{{\bf I}}

\newcommand{\VN}{\operatorname{VN}}
\newcommand{\Irr}{\operatorname{Irr}}
\newcommand{\SU}{\operatorname{SU}}
\newcommand{\cM}{\mathcal{M}}

\newcommand{\ignore}[1]{{ }}

\begin{document}

\maketitle
 \begin{abstract}
 Different notions of amenability on hypergroups and their relations are studied. Developing Leptin's theorem for discrete hypergroups, we characterize the existence of a bounded approximate identity for hypergroup Fourier algebras. We study  the Leptin condition for discrete  hypergroups derived from the representation theory of some classes of compact groups. Studying amenability of the hypergroup algebras for  discrete commutative hypergroups, we obtain  some results on amenability properties of some central  Banach algebras on  compact and discrete groups.
\vskip1.0em
{\bf Keyword:} hypergroups;    Fourier algebra; amenability;  compact groups; finite conjugacy groups.
\vskip1.0em
{\bf AMS codes:} 43A62,  46H20.

 \end{abstract}
 

\vskip2.5em
 
 In this paper, we investigate different amenability notions  of hypergroups and their relations. In particular, we look closer at a generalization of  F\o lner type conditions over hypergroups. This generalization  lets us to investigate more structural properties of  some classes of hypergroups. In particular we study the existence  of  bounded approximate identities  for the Fourier algebra  of  regular Fourier hypergroups.    We prove a hypergroup analogue of  Leptin's theorem for discrete regular Fourier hypergroups that is, the existence of a bounded approximate identity for the Fourier algebra is equivalent to the existence of a square integrable Reiter's  net. The later condition is  denoted by $(P_2)$.
 
 We also study amenability of hypergroup algebras (as  Banach algebras)  for discrete commutative hypergroups satisfying $(P_2)$. We show that for this class of hypergroups, the hypergroup algebra cannot be  amenable  if the inverse of the  Haar measure vanishes at  infinity.   This result and the appearance of $(P_2)$ in hypergroup Leptin's theorem emphasize the importance of the  amenability notion  $(P_2)$.   At the end we apply these results to some  examples of hypergroup structures  which are admitted by two classes of  locally compact groups, namely compact and discrete groups.

This paper is organized as follows. In Section~\ref{s:hypergroups}, we  introduce the notation  and present some preliminaries  regarding hypergroups and their Fourier spaces.  
  In an attempt to answer a question about Segal algebras on compact groups, the author, in \cite{ma}, defined the Leptin condition for  hypergroups.   In Section~\ref{ss:Folner-condition-on-^G} we expand this study by introducing more hypergroup  F\o lner type conditions.
The question of approximate amenability of Segal algebras on compact groups highlights the importance of the Leptin condition for the class of discrete hypergroups defined on the dual space of compact groups, \cite{ma}. In Subsection~\ref{ss:Leptin-number}, we apply some studies on the fusion rule of compact  groups to consider the Leptin condition for the  hypergroup structures admitted by some classes of compact groups including $\operatorname{SU}(n)$.
In Section~\ref{ss:bai-of-A(H)-Leptin-condition},  we characterize the existence of a bounded approximate identity in the Fourier algebra of discrete regular Fourier  hypergroups using the amenability notion $(P_2)$.  This characterization is an analogue of the classical {Leptin's theorem}.
In Section~\ref{s:AM-L1(H)}, we study the amenability of  the hypergroup algebras for discrete commutative hypergroups satisfying condition $(P_2)$. We follow this in   Subsections~\ref{ss:AM-ZA(G)} and \ref{ss:amenability-of-zl1(G)} where we apply  this result to two classes of central algebras on discrete and compact groups.
\vskip1.0em

Some results of this paper are  based on the work from the author's Ph.D. thesis, \cite{ma-the}, under the supervision of Yemon Choi and Ebrahim Samei.  

\vskip1.0em

\begin{section}{Preliminaries and notation}\label{s:hypergroups}

For the definition and properties of hypergroups, we refer the reader mainly to \cite{bl}.   Let  $H$ be a   \ma{hypergroup} with the identity  element $e$.  The notation $A*B$ stands for  $\cup\{\supp(\delta_x*\delta_y):\; \text{for all }\;x\in A, y\in B\}$ for  subsets $A,B$ of   $H$. By abuse of notation, we use $x*A$ to denote $\{x\}*A$.
 Let $C_c(H)$ and $C_0(H)$ respectively denote the space  of all compactly supported and vanishing at infinity complex valued continuous functions on $H$.
 For each $f\in C_c(H)$ and $x,y\in H$,  define $L_xf(y):=\delta_{\tilde{x}}*\delta_y(f)$. A Borel measure $\lambda$ on $H$ is called  a (left)  \ma{Haar measure} if  $\lambda(L_xf)=\lambda(f)$ for all $f\in C_c(H)$ and $x\in H$. In this paper we assume that  hypergroup $H$ always  possesses a Haar measure $\lambda$. The \ma{modular function} $\Delta:H \rightarrow \Bbb{R}^+$ satisfies  the identity $\lambda *\delta_{\tilde{x}}=\Delta(x)\lambda$ for every $x\in H$. 
The Banach spaces of all $p$-absolutely $\lambda$-integrable functions on $H$ are denoted by $L^p(H)$ for $1\leq p <\infty$. In particular  $L^1(H)$ forms a Banach algebra equipped with convolution
\[
f*_\lambda g(x):=\int_{H} f(t) L_{{t}}g(x) dh(t)\ \ \text{($f,g \in L^1(H)$)}
\]  
which is called the \ma{hypergroup algebra} of $H$.

Let $C(H)$ denote the set of all continuous bounded functions on $H$.
If $H$ is a commutative hypergroup, the \ma{dual of $H$}, denoted by $\widehat{H}$, is defined to be the set
$\{\alpha\in C(H)\;|\;  \alpha(\delta_x*\delta_y)=\alpha(x)\alpha(y), \alpha(\tilde{x})=\overline{\alpha(x)} \}$ equipped with the topology of uniform convergence on compact subsets of $H$. 
Here we use  $\varpi$ to denote the  \ma{Plancherel measure} on $\widehat{H}$.

\vspace{1.0em}

For a compact hypergroup $H$, Vrem, \cite{vr}, defined the \ma{Fourier space} similar to the Fourier algebra of a compact group. Subsequently, Muruganandam, \cite{mu1}, defined the \ma{Fourier Stieltjes space} on an arbitrary (not necessary compact) hypergroup $H$ using irreducible representations of $H$  similar to the group case.  Subsequently, he defined 
the { Fourier space} of a hypergroup $H$, as a closed subspace of the Fourier Stieltjes space, generated by $\{\xi *_\lambda\tilde{\xi}:\; \xi \in L^2(H)\}$. Therefore, $A(H) \cap C_c(H)$ is dense in $A(H)$ and further,
$A(H) \subseteq C_0(H)$,  $\norm{\cdot}_\infty \leq \norm{\cdot}_{A(H)}$, and for every $u\in A(H)$, $L_x u$, $\check{u}$, and $\overline{u}$ belong to $A(H)$.
A hypergroup $H$ is called a \ma{regular Fourier hypergroup}, if the Banach space ${A}(H)$ equipped with the pointwise multiplication  is a Banach algebra. 

For each $f \in L^1(H)$ and $\xi \in L^p(H)$ ($1\leq p \leq \infty)$, it is known that $f *_\lambda \xi \in L^p(H)$ with $\norm{f *_\lambda \xi}_p \leq \norm{f}_1 \norm{\xi}_p$. Define the operator $\lambda_f$ by $\lambda_f(\xi)=f*_\lambda\xi$ for every $f \in L^1(H)$ and $\xi \in L^2(H)$. The von Neumann algebra constructed by $\{\lambda_f: f\in L^1(H)\}$ in ${\cal B}(L^2(H))$ is denoted by $\VN(H)$  and called the \ma{von Neumann} algebra of $H$.
By \cite[Theorem~2.19]{mu1},   $A(H)$ is the predual of $\VN(H)$.
The $C^*$-algebra generated by $\{\lambda_f: f\in L^1(H)\}$ in ${\cal B}(L^2(H))$ is called the \ma{reduced $C^*$-algebra} of $H$ and denoted by $C^*_\lambda(H)$ which is a $C^*$-subalgebra of $VN(H)$.  Moreover, $A(H)$ can be considered as a subalgebra of $B_\lambda(H)$ where $B_\lambda(H)$ is the dual of $C^*_\lambda(H)$. Note that for  $f \in L^1(H) \subseteq C^*_\lambda(H)$  and $u \in B_\lambda(H)$, the dual product is just integration of the pointwise product $uf$ with respect to $\lambda$.
In this paper we rely on the following lemma which we present from \cite{ma} without its proof.

\begin{lemma}\label{l:A(H)-properties}\cite[Lemma~3.4]{ma}\\
Let $H$ be a hypergroup, $K$ a compact subset  of $H$ and $U$ an open subset of $H$ such that $K\subset U$. Then for each relatively compact open set $V$  such that $\overline{K *V*\check{V}} \subseteq U$,  $u_V:={\lambda(V)}^{-1} 1_{K*V} *_\lambda  \tilde{1}_{V}$ belongs to $A(H) \cap C_c(H)$. Also $u_V(H)\geq 0$, $u_V|_K\equiv 1$, $\supp(u_V) \subseteq U$, and 
\[
\norm{u_V}_{A(H)} \leq \left(\frac{ \lambda(K*V)}{{\lambda(V)}}\right)^{\frac{1}{2}}.
\]
\end{lemma}

\begin{rem}\label{r:existence-of-the-V}
For each pair $K,U$ such that $K \subset U$,  we can always find  a relatively compact neighborhood  $V$ of the identity, $e$,  that satisfies the conditions in Lemma~\ref{l:A(H)-properties}. The existence of such $V$ is a result of continuity of the mapping $(x,y)\mapsto x*y$ with respect to the locally compact topology of $H\times H$ into the Michael topology on ${\textfrak C}(H)$ (see \cite{bl}). Since $H$ is locally compact, there exists some relatively compact open set $W$ such that $K\subseteq W \subseteq \overline{W} \subseteq U$; $K\in {\textfrak C}_{H\setminus \overline{W}}(W)$ as an open set in the Michael topology and consequently for each $x\in K$, $x*e\in {\textfrak C}_{H\setminus \overline{W}}(W)$. Since, the mapping $e \rightarrow x*e$ is continuous, there is some neighborhood $V_1^x$ of $e$ such that for each $y\in V_1^x$, $x*y\in {\textfrak C}_{H\setminus \overline{W}}(W)$ i.e. $x*y \subseteq W$ and $x*y \cap H\setminus \overline{W}=\emptyset$. Let us define 
$V^{(1)}=\cup_{x\in K} (V_1^x \cap \check{V}_1^x)$.
Clearly, $\check{V}^{(1)}=V^{(1)}$. Moreover, 
$
K*V^{(1)}=\cup_{y\in V^{(1)}}\cup_{x\in K} x*y \subseteq \cup_{x\in K} x*V_1^x \subseteq W
$
and $K*V^{(1)} \cap H\setminus \overline{W}=\emptyset$ since $(x*y) \cap (H\setminus  \overline{W})=\emptyset$ for all $x\in K$ and $y\in V^{(1)}$. Now let us replace $K$ by the compact set $\overline{K*V^{(1)}}$. Similar to the previous argument, for some relatively compact open set $W'$ such that $\overline{K*V^{(1)}} \subseteq W'\subseteq \overline{W'} \subseteq U$, one can find some $V^{(2)}$ a neighborhood of $e$ such that $V^{(2)}=\check{V}^{(2)}$, $\overline{K*V^{(1)}}*V^{(2)} \subseteq W'$, and $( \overline{K * V^{(1)}} * V^{(2)}) \cap (H\setminus \overline{W'} )=\emptyset$. Hence, for the relatively compact open set $V:=V^{(1)}\cap V^{(2)}$, one gets that $V=\check{V}$ and
$
{K*V*\check{V}} \subseteq \overline{K*V^{(1)}} * V^{(2)} \subseteq \overline{W'}$.
So $\overline{K * V* \check{V}} \subseteq U$.
\end{rem}

In \cite{mu1}, Muruganandam showed that when $H$ is commutative, $A(H)=\{\xi *_\lambda  \tilde{\eta}:\; \xi, \eta\in L^2(H)\}$
 and  $\norm{u}_{A(H)}=\inf  \norm{\xi}_2 \norm{\eta}_2$ for all $\xi, \eta \in L^2(H)$ such that  $u=\xi *_\lambda  \tilde{\eta }$.
 A similar characterization is known for the Fourier space of compact hypergroups, \cite{vr}.
The following implies that this fact is true for all unimodular hypergroups.

\begin{prop}\label{p:Fourier-of-hypergroups}
Let $H$ be a unimodular hypergroup. Then $A(H)=\{\xi*\tilde{\eta}: \xi, \eta \in L^2(H)\}$ and $\norm{u}=\inf\{\norm{\xi}_2\norm{\eta}_2\}$ over all $\xi, \eta \in L^2(H)$ with  $u=\xi *\tilde{\eta}$ where the infimum is  attained.
\end{prop}

A reference for this  characterization of  Fourier algebras on  locally compact groups is the Master's thesis of Zwarich, \cite{zw-the}. Chapter~4   is dedicated to this proof for locally compact groups based on an observation by Haagerup, \cite{haa}, and the von Neumann theory developed in \cite{dix2}. The proof, in \cite{zw-the}, is written based on the properties of von Neumann algebras which are in the \ma{standard form} as defined in \cite{haa}. This argument can be used easily for unimodular  hypergroups as well. 
Before we present the proof, let us recall that every element  $u\in A(H)$ acts $\sigma$-weak continuously on $\VN(H)$; hence,  $u$ is actually a \ma{normal} linear functional on $\VN(H)$, \cite[Section~3.3]{zw-the}.  Note that for non-unimodular hypergroups, condition (v) of \cite[Definition~4.3.3]{zw-the} may not be satisfied. 

\vskip0.75em

\begin{proof}
For $\xi, \eta \in C_c(H)$, define $\xi^*(x):= {\overline{\xi(\tilde{x})}}$, $ \xi^{\sharp}(x):= \xi(x)$ and the canonical inner product 
\[
\langle \xi, \eta \rangle := \int_H \xi(x) \overline{\eta(x)} dx. 
\]
Then $C_c(H)$ forms a \ma{quasi-Hilbert algebra}  as defined in \cite[Definition~4.3.3]{zw-the}. One can re-write the proof of \cite[Proposition~4.3.4]{zw-the}  with appropriate modifications. Obviously the Hilbert space generated by $C_c(H)$ would be $L^2(H)$. Further, the second commutant of $\{\lambda_f: f \in C_c(H)\}$ is $\VN(H)$, (see \cite[Remark~2.18]{mu1}). Therefore, by \cite[Theorem~4.3.11]{zw-the}, $\VN(H)$ is in standard form. Hence, one  applies  \cite[Theorem~4.3.16]{zw-the}    to the elements of $A(H)$ as normal linear functionals  on $\VN(H)$  and conclude that $A(H)=\{\xi*\tilde{\eta}: \xi, \eta\in L^2(H)\}$ and the condition of the norm is satisfied. 
\end{proof}

\end{section}

 

\begin{section}{F\o lner type conditions on Hypergroups}\label{ss:Folner-condition-on-^G}


 Amenable locally compact groups are characterized by a variety of
properties including F\o lner type conditions.
These conditions relate the concept of ``amenability" (which is an algebraic notion on the group algebra) to some structural properties of the underlying group.
 In this section, we look at a generalization of  F\o lner type conditions over hypergroups.

 In \cite{ma}, the author introduced the Leptin condition for hypergroups. Here, we define more F\o lner type conditions for hypergroups and we study their relations.  
 For each two subsets $A$ and $B$ of some set $X$, we denote their symmetric difference, $(A\setminus B) \cup (B\setminus A)$, by $A\triangle B$.

\begin{dfn}\label{d:Folner-Leptin-condition}
Let $H$ be a hypergroup and $D\geq 1$ an integer. 
\begin{itemize}
\item[$(L_D)$]{We say that $H$ satisfies the \ma{$D$-Leptin condition} if for every  compact subset $K$ of $H$ and $\epsilon>0$, there exists a measurable set $V$ in $H$ such that $0<\lambda(V)<\infty$ and $\lambda(K*V)/\lambda(V) < D + \epsilon$.}
\item[$(F)$]{ We say that $H$ satisfies the \ma{F\o lner condition} if for every  compact subset $K$ of $H$ and $\epsilon>0$, there exists a measurable set $V$ in $H$ such that
$0<\lambda(V)<\infty$ and $\lambda (x*V\triangle V)/\lambda(V) < \epsilon$ for every $x\in K$.}
\item[$(SF)$]{ We say that $H$ satisfies the \ma{Strong F\o lner condition} if for every  compact subset $K$ of $H$ and $\epsilon>0$, there exists a measurable set $V$ in $H$ such that
$0<\lambda(V)<\infty$ and $\lambda(K*V\triangle V)/\lambda(V) < \epsilon$.}
\end{itemize}
\end{dfn}

Immediately, every compact hypergroup $H$ satisfies all conditions $(SF)$, $(F)$, and $(L_1)$.
The proof is a direct result of finiteness of the Haar measure on compact hypergroups, \cite{bl}, by setting  $V=H$ for all conditions in Definition~\ref{d:Folner-Leptin-condition}.

\begin{rem}\label{r:relatively-compact}
In the definition of  the $D$-Leptin condition, $(L_D)$,  one can suppose that $V$ is compact without loss of generality. To show this,  suppose that $H$ satisfies the $D$-Leptin condition. For compact subset $K$ of $H$ and $\epsilon>0$, there exists a measurable set $V$ such that  $\lambda(K*V)/\lambda(V) < D + \epsilon$. Using regularity of $\lambda$, as a measure, for each positive integer $n$, we can find compact set $V_1\subseteq V$ such that $\lambda(V\setminus V_1) < \lambda(V)/n$. This implies that $0< \lambda(V_1)$ and  $ \lambda(V)/\lambda(V_1) < {n}/{(n-1)}$. Therefore
\[
\frac{\lambda(K*V_1)}{\lambda(V_1)} \leq \frac{\lambda(V)}{\lambda(V_1)} \left( \frac{\lambda(K*V_1)}{\lambda(V)}\right)< \frac{n}{n-1}(D+\epsilon).
\]
\end{rem}

\begin{proposition}\label{p:Strong-Folner-implies-Folner-&-Leptin}
For every hypergroup $H$,   $(SF)$  implies  $(L_1)$. Further, if $H$ is discrete,   $(F)$ implies $(SF)$ and consequently $(F)$ implies $(L_1)$.
\end{proposition} 

\begin{proof}
$(SF) \Rightarrow (L_1)$. For a compact set $K$ and $\epsilon>0$, let $V$ be a measurable set such that $\lambda(K*V \triangle V) < \epsilon \lambda(V)$.
Hence
\begin{eqnarray*}
 \frac{\lambda(K*V)}{\lambda(V)} -1 &=& \frac{\lambda(K*V)  - \lambda(V)}{\lambda(V)}\\
&\leq&  \frac{\lambda(K*V) + \lambda(V) -2 \lambda((K*V) \cap V)}{\lambda(V)} = \frac{\lambda((K*V)\triangle V)}{\lambda(V)} <\epsilon .
\end{eqnarray*}

$(F) \Rightarrow (SF)$.  Suppose that $H$ is discrete. 
Let $K$ be a non-empty finite subset of $H$; $K= \{x_i\}_{i=1}^n$. Therefore, for each $\epsilon>0$ there is a finite set $V$ such that $0<\lambda(V)$ and 
\[
\frac{\lambda((x*V)\triangle V) }{\lambda(V)}<\frac{\epsilon}{|K|}\ \ \ (x\in K).
\]
So
\begin{eqnarray*}
\frac{\lambda( (\bigcup_{i=1}^n x_i) *V \triangle V)}{\lambda(V)}  =  \frac{\lambda(\bigcup_{i=1}^n ( x_i *V) \triangle V)}{\lambda(V)}\leq  \sum_{i=1}^n  \frac{\lambda(x_i*V \triangle V)}{\lambda(V )} \ \ \ = \epsilon.
\end{eqnarray*} 
The last inequality is a result of the following inclusion about arbitrary sets $B_1, B_2, C$:
\begin{eqnarray*}
((B_1\cup B_2)\triangle C) \subseteq \big(B_1 \triangle C)\big) \cup \big(B_2\triangle C)\big).
\end{eqnarray*}
\end{proof}

If $H$ is a locally compact group, all the conditions $(F)$, $(SF)$, and $(L_1)$ are equivalent to each other and to   amenability of the group $H$. 
Trying to adapt the rest of the relations  between $(F)$, $(SF)$, and $(L_1)$ from the group case, \cite{pi},  one may notice that in almost all of the arguments, the inclusion $x(A\setminus B)\subseteq xA \setminus xB$ is crucially applied  where $A,B$ are subsets of the group $H$ and $x$ is one arbitrary element. This inclusion  is not necessarily true for a general hypergroup though.

In \cite{la4} the notion of summing sequences in the context of polynomial hypergroups is introduced to study some F\o lner type properties. It is straightforward to check that a polynomial hypergroup $\Nat_0$ with a summing sequence $(A_n)_{n\in \Bbb{N}_0}$  satisfies all the { $1$-Leptin}, { Strong F\o lner}, and F\o lner conditions.

\begin{subsection}{$D$-Leptin condition on  dual of compact groups}\label{ss:Leptin-number}

Let $G$ be a compact group and $\Irr(G)$ the set of all equivalent classes of irreducible unitary representations of $G$. It is known that $\Irr(G)$ forms a discrete commutative regular Fourier hypergroup with a Haar measure $\lambda$ where $\lambda(\pi)=d_\pi^2$ for $d_\pi$ the dimension of the representation $\pi$, look at \cite[1.1.14]{bl}.
In \cite[Section~4]{ma}, it was shown that for $G=\SU(2)$, the compact group of $2\times 2$   unitary matrices,   $\Irr(G)$ satisfies the  $1$-Leptin condition. Here we show that  more compacts groups  enjoy these conditions.
The  proof  of the following proposition is similar to \cite[Proposition~4.4]{ma} so it is omitted here.

\begin{proposition}\label{p:Leptin-of-product-groups}
Let $G=\prod_{i\in\Ind} G_i$ for a family of compact groups $(G_i)_{i\in \Ind}$ such that for each $i\in\Ind$, $\widehat{G}_i$ satisfies the $D_i$-Leptin condition. Then if $D:=\prod_{i\in \Ind} D_i$ exists, $\Irr(G)$ satisfies the $D$-Leptin condition.  
\end{proposition}

Let $G$ be a connected simply connected compact real Lie group. Then, $\Irr(G)$, as the dual object of a compact Lie group,  forms a finitely generated hypergroup. Suppose that $F$ is a finite generator of $\Irr(G)$;
therefore, by \cite[Theorem~2.1]{ve}, there exists positive integers $0<\alpha,\beta <\infty$ such that
\begin{equation}\label{eq:growth-in-dual-of-Lie-groups}
\alpha \leq \frac{\lambda(F^k)}{k^{d_G}} \leq \beta
\end{equation}
for all $k\in\Bbb{N}$  where $d_G$ is the dimension of  the group $G$ as a Lie group over $\Bbb{R}$.  This estimation for the growth rate of $\Irr(G)$  leads to the  $D$-Leptin condition for $\Irr(G)$ as follows.

\begin{cor}\label{c:D-Leptin-of-simply-connected-Lie}
Let $G$ be a connected simply connected compact real Lie group. Then $\Irr(G)$, as a hypergroup, satisfies the $D$-Leptin condition for some $D\geq 1$.
\end{cor}

\begin{proof}
Let $K$ be a finite subset of $\Irr(G)$. 
Suppose that $F$ is a finite generator of $\Irr(G)$. For some $k\in \Bbb{N}$, $K\subseteq F^k$.  By applying (\ref{eq:growth-in-dual-of-Lie-groups}),  we have
\begin{eqnarray*}
\limsup_{\ell\rightarrow \infty} \frac{\lambda(K*F^\ell)}{\lambda(F^\ell)} &\leq& \limsup_{\ell\rightarrow\infty}\frac{\lambda(F^{\ell+k})}{\lambda(F^\ell)} =
 \limsup_{\ell \rightarrow \infty} \frac{\lambda(F^{\ell+k})}{(\ell+k)^{d_{{G}}}} \;  \frac{\ell^{d_{{G}}}}{\lambda(F^\ell)}\; \frac{(\ell+k)^{d_G}}{\ell^{d_G}}\leq \beta/\alpha.
\end{eqnarray*}
\end{proof}

Let  $\operatorname{SU}(3)$ denote the compact group of $3\times 3$  unitary matrices which is  a connected simply connected compact real  Lie group. 
Here we apply some studies on the representation theory of real connected Lie groups to find a concrete answer for the $D$-Leptin  property of $\Irr(\SU(3))$.
 The idea of the proof is similar to the one in the proof of \cite[Theorem~2.1]{ve}.
 
\begin{proposition}\label{p:Leptin-number-of-SU(3)}
The hypergroup $\Irr(\SU(3))$ satisfies the $18240$-Leptin condition.
\end{proposition}
 
 \begin{proof}
 Let us follow \cite{bour9} in the notation and basic facts about compact Lie groups $G$ and in particular $G=\SU(3)$. 
  Let  the set of  all fundamental weights $\beta$ be denoted by $B$. Then we have  $(\beta| {\beta'}) = 0$ for $\beta\neq \beta'$ while $ (\beta|\beta) > 0$ for all $\beta, \beta'\in B$. Taking highest weights induces an identification between the set $\Irr(G)$ of classes of irreducible unitary representations and the set of dominant weights $X_{++}$  which are all $( p_\beta \beta)_{\beta}$ for $p_\beta \in \Bbb{N}_0=\{0,1,2,\ldots\}$. From now on, without loss of generality, we denote the element $\pi \in \Irr(G)$  by its corresponding multipliers in $X_{++}$ that is $\pi=( p_\beta)_{\beta}$.
It is  known  that in the case of connected simply connected compact real Lie groups, the set $F$, of representations $\delta_{\beta_0}$ which is $\delta_{\beta_0}=( p_\beta)_{\beta}$ where $p_\beta=0$ for all $\beta\neq \beta_0$ and $p_{\beta_0}=1$, forms a generator of $\Irr(G)$. 
Further, one may define a mapping $\tau : \Irr(G) \rightarrow \Bbb{N}_0$ where $\tau( \pi)=\sum_\beta p_\beta$ such that for each $\pi=( p_\beta)_{\beta}$, $\pi$  belongs to $F^{\tau(\pi)} \setminus F^{\tau(\pi)-1}$.
The dimension of $\pi=(p_\beta)_\beta$ is given by Weyl's formula 
\[
d_\pi = \prod_{\alpha \in R^+} \left( 1 +\frac{\sum_{\beta} p_\beta (\beta|\alpha)}{(\rho | \alpha)}\right)
\]  
where $\rho$ is the sum of the fundamental dominant weights where $(\rho,\beta)=1$ for each $\beta$.

Restricting this observation to the case  $G=\operatorname{SU}(3)$, one gets that $\Irr(G)$ is nothing but the set of all $\pi=(p,q)$ where $p,q\in \Bbb{N}_0$ while $d_\pi=(p+1)(q+1)(2+p+q)/2$.  Also according to  the finite generator $F=\{(0,1),(1,0)\}$, one gets that $S_k:=F^k\setminus F^{k-1}$ is nothing but the set $\{(k,k-j)\}_{j=0}^k$. Hence, based on some  straightforward computations, one gets that
\[
 \lambda(S_k)=\sum_{j=0}^k \frac{(j+1)^2 (k-j+1)^2 (k+2)^2}{4}.
\]
One may use this fact that $F^k=\sum_{j=0}^k S_j$ to get
\[
\frac{\lambda(F^n)}{n^8}= \frac{1}{n^8}+\frac{3 n^7+60 n^6+518 n^5+2520 n^4+7547 n^3+14220 n^2+16412 n+10560}{2880 n^7}.
\]
Therefore, $ {1}/{960} <  {\lambda(F^n)}/{k^n}  \leq 19$.
 Now the argument mentioned in the proof of Corollary~\ref{c:D-Leptin-of-simply-connected-Lie} implies that $\Irr(G)$ satisfies the $18240$-Leptin condition.
 \end{proof}

\begin{rem}
In \cite{ma-the}, the author applied a study on the tensor decomposition of irreducible representations of $\operatorname{SU}(3)$, \cite{we}, to show that $\Irr(\SU(3))$ satisfies  $3^8$-Leptin condition which is   smaller than  the amount found in in Proposition~\ref{p:Leptin-number-of-SU(3)}.  However, since the proof of Proposition~\ref{p:Leptin-number-of-SU(3)} has  structural details which allow similar computations for other $\SU(n)$'s, we found it interesting, (although these doable computations would be tedious   even  for $n$ as small as $3$). 
\end{rem}

\end{subsection}

\end{section}

\vskip2.0em

\begin{section}{Bounded approximate identity of Fourier algebra}\label{ss:bai-of-A(H)-Leptin-condition}
 
Let  $H$  be a  regular Fourier hypergroup and  $D\geq 1$.   By $(B_D)$ we mean  the existence of an  approximate identity of $A(H)$ whose $\norm{\cdot}_{A(H)}$-norm is bounded by   $D$  which we  call    a \ma{$D$-bounded approximate identity}.

A    commutative hypergroup $H$ is called a \ma{strong hypergroup} if   $\wH$, as the dual of $H$, is a hypergroup whose Haar measure corresponds the Plancherel measure. A strong  hypergroup $H$ is  regular Fourier  and satisfying $(B_1)$.  This observation is based on this fact that $A(H)$ is isometrically Banach algebra isomorphic to the hypergroup algebra $L^1(\wH)$ through the Fourier transform, \cite[Proposition~4.2]{mu1}, and since it is known that hypergroup algebras have $1$-bounded approximate identities. 

In this section we study $(B_D)$ for general regular Fourier hypergroups.

\begin{proposition}\label{p:Leptin->bai of A(H)}
Let $H$ be a regular Fourier hypergroup which satisfies the $D$-Leptin condition, $(L_D)$, for some $D\geq 1$. Then $A(H)$ has a $D$-bounded approximate identity, $(B_D)$.
\end{proposition}

\begin{proof}
Fix $\epsilon>0$. Using the $D$-Leptin condition on $H$,  for every arbitrary non-void compact set $K$ in $H$, we can find a finite subset $V_K$ of $H$ such that $\lambda(K*V_K)/\lambda(V_K)<D^{2}(1+\epsilon)^2$. Using Lemma~\ref{l:A(H)-properties},  for 
\[
v_K:=\frac{1}{\lambda(V_K)} 1_{K *V_K} *_\lambda   \tilde{1}_{V_K}
\] 
we have  $\norm{v_K}_{A(H)} < D(1+\epsilon)$ and $v_{K}|_{K}\equiv 1$. Define for each pair $(K,\epsilon)$, $a_{\epsilon, K}=(1+\epsilon)^{-1}v_K$. 
\noindent We consider the net 
$\{a_{\epsilon,K}: K\subseteq H\ \text{compact, and $0<\epsilon<1$}   \}$ in $A(H)$ where $a_{\epsilon_1,K_1} \preccurlyeq a_{\epsilon_2,K_2}$ whenever $v_{K_1}v_{K_2} = v_{K_1}$ and $\epsilon_2 < \epsilon_1$.
So $(a_{\epsilon,K})_{0<\epsilon<1, K\subseteq H}$ forms a $\norm{\cdot}_{A(H)}$-norm $D$-bounded net in $A(H) \cap C_c(H)$.
 Let $u\in A(H)\cap C_c(H)$ with $K=\supp(u)$. Then $v_K u= u$. Therefore, 
  $(a_{\epsilon,K})_{0<\epsilon<1, K\subseteq H}$ is a $D$-bounded approximate identity of $A(H)$.
\end{proof}

Reiter's condition $(P_p)$ $(1\leq p < \infty$)  for  hypergroups first was  defined  in \cite{sk}.
For $1\leq p <\infty$, we say that $H$ satisfies $(P_p)$,  if whenever $\epsilon > 0$ and a compact set $K \subseteq H$ are
given, then there exists  $\xi\in  L^p(H)$ so that $\xi \geq 0$, $\norm{\xi}_p=1$, and 
$
\norm{\delta_x*\xi - \xi}_r <\epsilon$ for every $x\in K$.
For every hypergroup $H$, $(P_1)$ is equivalent to the amenability of $H$. Although $(P_2)$ implies $(P_1)$, \cite{sk},
 $(P_2)$ is not necessarily equivalent to the amenability of the hypergroup. 
Singh, \cite[Proposition~4.4.3]{singh-mem}, showed that if a hypergroup $H$ satisfies the ($1$-)Leptin condition, it satisfies $(P_p)$  for any $p\in [1,\infty]$. For Reiter's conditions on Fusion algebras and their consequences look at \cite{izu}.

In the following we crucially rely on \cite[Lemma~4.4]{sk} which proves that  $H$ satisfies $ (P_2)$ if and only if 
there is a net $(\xi_\alpha)_\alpha\subseteq L^2(H)$ such that $\norm{\xi_\alpha}_2=1$ and $\xi_\alpha* \tilde{\xi}_\alpha$
converges to $1$ uniformly on compact subsets of $H$.
Indeed, $(P_2)$ implies the existence of a net $(u_\alpha)$  (in the form of $u_\alpha:=\xi_\alpha*\tilde{\xi}_\alpha$)  in  $A(H)$ with $\norm{u_\alpha}_{A(H)} \leq \norm{\xi_\alpha}_2^2 =1$ which converges to $1$ uniformly on compact sets. 

\vskip1.0em

To prove the main  result of this section, we need the following lemma. 
 The  proof is inspired by  \cite[Corollary~3.5]{ru-2d}.
 
\begin{lem}\label{l:ideal}
Let $H$ be a regular  Fourier hypergroup. Then $A(H)$ is an  ideal in its second dual if  $H$ is discrete.
\end{lem}
 
\begin{proof}
First let $H$ be a discrete hypergroup. 
Let $J$ be the closed subspace of $A(H)$ of all functions $u$ such that $v\mapsto vu$ ($A(H) \rightarrow A(H)$) is weakly compact. 
Note that for each $u\in L^2(H)\subseteq A(H)$,   $\norm{uv}_{A(H)} \leq \norm{u}_2\norm{v}_{A(H)}$, hence the mapping $v\mapsto vu$ ($A(H) \rightarrow L^2(H)$) is weakly compact, because  $L^2(H)$, as a Hilbert space, is reflexive. Also note that $\norm{\cdot}_{A(H)}\leq \norm{\cdot}_2$; hence, the inclusion $L^2(H) \hookrightarrow A(H)$ is a norm decreasing mapping. Therefore, combining these two mappings, for each $u\in L^2(H)$, $v \mapsto v u$ ($A(H) \rightarrow A(H)$) is weakly compact.
Hence, $L^2(H) \subset J$ while $L^2(H)$ is dense in $A(H)$ (as $c_c(H)$ is dense in   $A(H)$). Therefore, $J=A(H)$. But \cite[Proposition~1.4.13]{pa1} says that such a Banach algebra is an ideal in its second dual. 
\end{proof}

 The converse of Lemma~\ref{l:ideal} is also true. A proof based on character amenability (of Fourier algebras of hypergroups) from the unpublished manuscript \cite{ma10} was drawn to our attention by M. Nemati.

\vskip1.0em

The following theorem resembles   the Leptin theorem for  discrete regular Fourier hypergroups. 
In the proof, some techniques of the  group case (see \cite[Theorem~7.1.3]{ru}) 
have been applied.
 Recall that a \ma{state} on a $C^*$-algebra is a positive linear functional of norm $1$. Moreover, if $\cM$ is a von Neumann algebra with predual $\cM_*$, every state of $\cM$ can be approximated by a net of states in $\cM_*$  in  the weak$^*$ topology.  

\begin{theorem}\label{t:bai-A(H)<=>P-2}
Let $H$ be a  discrete  regular Fourier hypergroup. Then 
the following conditions are equivalent.
\begin{itemize}
\item[$(B_1)$]{ $A(H)$ has a $1$-bounded  approximate identity.}
\item[$(B_D)$]{ $A(H)$ has a $D$-bounded approximate identity for some $D\geq 1$.}
\item[$(P_2)$]{ $H$ satisfies $(P_2)$.}
\end{itemize}
\end{theorem}

\begin{proof}
$(B_1)\Rightarrow (B_D)$ is trivial.

$(B_D) \Rightarrow (P_2)$. 
 For $(e_\alpha)_\alpha$ a $D$-bounded approximate identity of $A(H)$, there exists a $w^*$-cluster point $F\in \VN(H)^{*}$. By replacing $(e_\alpha)_\alpha$ with a subnet if necessary, note that for each $x\in H$, $
\langle\lambda_x ,F\rangle = \lim_\alpha \langle\lambda_x ,e_\alpha\rangle = \lim_\alpha e_\alpha(x)=1$.
So $F|_{L^1(H)}$ may be interpreted as the constant function $1$ on $H$.
Hence $F|_{L^1(H)}$ is a multiplicative functional on $L^1(H)$ (i.e. $ \langle F, f*g\rangle  = \langle F,f\rangle\; \langle F,g\rangle$ for $f,g \in L^1(H)$). But $L^1(H)$ is dense in   $C^*_\lambda(H)$; hence,  $F|_{C^*_\lambda(H)}$ is also a (positive) multiplicative functional on $C^*_\lambda(H)$. Therefore, $F_0=F|_{C^*_\lambda(H)}$  is a state on $C^*_\lambda(H)$. 

Thus by 
\cite[Corollary~2.3.12]{ren},
$F_0$ is extendible to a state $E$ on $\VN(H)$.
Because  states of $\VN(H)$ which belong to $A(H)$ are weak$^*$ dense in the set of all states of $\VN(H)$, we may find a net $(u_\beta)_\beta$ in $\{ \xi *_\lambda   {\tilde{\xi}}:\ \xi\in L^2(H)\}$ such that $u_\beta=\xi_\beta *_\lambda   {\tilde{\xi_\beta}} \rightarrow E$ in  the weak$^*$ topology where  $(\xi_\beta)_\beta \subseteq L^2(H)$. 
Moreover, 
\[
\norm{\xi_\beta}_2^2= u_\beta(e) \leq  \norm{u_\beta}_{A(H)} (=1) \leq \norm{\xi_\beta}_2^2.
\] 
Recall that $E|_{L^1(H)}=F|_{L^1(H)} $. Also for each $f \in L^1(H)$ and $u\in A(H)$, $uf \in L^1(H)$. Therefore,
\[
\lim_\beta \langle uu_\beta, f\rangle = \langle E, uf\rangle =\langle F_0, uf\rangle = \lim_\alpha \langle e_\alpha, uf\rangle = \langle u, f\rangle.
\]
Therefore, $uu_\beta \rightarrow u$ with respect to the  topology $\sigma(A(H), L^1(H))$. This implies that $u_\beta(x)  \rightarrow 1$ on finite subsets of $H$ as $\delta_x \in L^1(H)$ for every $x\in H$. So by   \cite[Lemma~4.4]{sk}, the existence of the bounded  net $(\xi_\beta)_\beta$ implies $(P_2)$.

\vskip1.0em

$(P_2)\Rightarrow (B_1)$. Let $(u_\beta)_\beta$ be the net generated by $(P_2)$ in \cite[Lemma~4.4]{sk}, that is $u_\beta=\xi_\beta*\tilde{\xi}_\beta$ for some $\xi_\beta\in L^2(H)$ while $\norm{\xi_\beta}_2=1$ for every $\beta$ and $u_\beta\rightarrow 1$ uniformly on finite sets. Therefore,
$\norm{u_\beta}_{A(H)} \leq \norm{\xi_\beta}_2^2 = 1$ for each $\beta$.
As a bounded net in $A(H)$, by possibly reducing to a subnet,  $u_\beta \rightarrow E$ in the weak$^*$ topology of $\VN(H)^*$ for some $E \in \VN(H)^*$. Moreover, since $u_\beta(x) \rightarrow 1$ for every $x\in H$,   $E|_{L^1(H)}$ acts as the constant function $1$ where $L^1(H)$ is considered as a subalgebra of $\VN(H)$.

   Let $T$ be an arbitrary element in $\VN(H)$ and $(f_\alpha)_\alpha$ a net in $L^1(H)$ so that $f_\alpha \rightarrow T$ in the weak$^*$ topology on $\VN(H)$.  Let $u$ be an arbitrary element in $A(H)$. Recall that $uE \in A(H)$, by Lemma~\ref{l:ideal}. Therefore, we get
\[
\lim_\beta \langle uu_\beta , T\rangle =  \langle uE , T \rangle = \lim_\alpha \langle uE, f_\alpha \rangle = \lim_\alpha \langle u, f_\alpha \rangle = \langle u, T\rangle.
\]
Consequently,   $uu_\beta \rightarrow u$  in the weak topology $\sigma(A(H), \VN(H))$.
It is a well-known result of functional analysis that the weak closure of a convex set coincides with its norm closure. So one can render a bounded net $(v_\alpha)_\alpha$ in $\overline{\operatorname{conv}}\{u_\beta: \beta \}$  so that for each $v \in A(H)$, $\norm{vv_\alpha - v}_{A(H)} \rightarrow 0$.
\end{proof}

\ignore{

\vskip10.0em

Also as a multiplicative functional,  $\norm{F|_{C^*_\lambda(H)}}=1$. But as a positive norm $1$ functional, $F|_{C^*_\lambda(H)}$ is a state. 
Thus,
by \cite[Corollary~2.3.12]{ren},
$F|_{C^*_\lambda(H)}$ is extendible to a state $E$ on $\VN(H)$.
Because  states of $\VN(H)$ which belong to $A(H)$ are weak$^*$ dense in the set of all states of $\VN(H)$, we may find a net $(f_\beta)_\beta$ in $\{ f*_\lambda   {\tilde{f}}:\ f\in L^2(H,h)\}$ such that $f_\beta=g_\beta *_\lambda   {\tilde{g_\beta}} \rightarrow E$ in  weak$^*$ topology for a net $(g_\beta)_\beta \subseteq L^2(H,h)$. 
Moreover, 
\[
\norm{g_\beta}_2^2= f_\beta(e) \leq  \norm{f_\beta}_{A(H)} (=1) \leq \norm{g_\beta}_2^2.
\] 
 Let $T \in \VN(H)$ and $(h_\gamma)_\gamma \subset L^1(H)$ so that $h_\gamma \rightarrow T$ in the weak$^*$ topology. Recall that $E|_{L^1(H,h)}$ is the constant function $1$. Then for each $u \in A(H)$ we have
\begin{eqnarray*}
\lim_\beta \langle u f_\beta, T\rangle = \langle u E, T\rangle = \lim_\gamma \langle u E, h_\gamma \rangle = \lim_\alpha \langle 1, u h_\gamma \rangle = \lim_\gamma \langle u,h_\gamma \rangle = \langle u, T\rangle.
\end{eqnarray*}
Therefore, $uu_\beta \rightarrow u$ with respect to the weak topology $\sigma(A(H), \VN(H))$.
It is a well-known result of functional analysis that the weak closure of a convex set coincides with its norm closure, so  for every $\epsilon > 0$, there exists $\varphi_{\{u_1,\ldots, u_n\},\epsilon} = \varphi \in \operatorname{conv}\{f_\beta\}$ such that $u_i\in A(H)$ for $i = 1,\ldots,n$ and 
$\norm{u_i \varphi - u_i}_{A(H)}<\epsilon$.
Moreover, 
\[
1=\varphi(e) \leq \norm{\varphi}_\infty \leq \norm{\varphi}_{A(H)} \leq 1.
\]
Note that $\varphi$ is also is a positive functional in the cone of positive functionals on $\VN(H)$; therefore, $\varphi$ is actually a state and $\varphi=\psi*\tilde{\psi}$ for some $\psi \in L^2(H)$.

To make the set of all such $\varphi$'s a net, let $I := \{(S,\epsilon): S \subseteq A(H) \ \text{is finite},\ \epsilon> 0\}$ become a directed set by $(S,\epsilon) \leq (S',\epsilon')$ if $S\subseteq S'$ and $\epsilon \geq \epsilon'$. This lets us to render the  net $(\varphi_\alpha)_\alpha \subseteq \operatorname{conv}\{f_\beta\}$
that is a bounded  approximate identity of $A(H)$. On the other hand, for each compact set $K\subseteq H$, by Lemma~\ref{l:A(H)-properties}, there is some $u_K\in A(H)$ such that $u_K|K\equiv 1$. 
Therefore, for each $x\in K$,
\begin{eqnarray*}
\lim_\alpha|1-\varphi_\alpha(x)| = \lim_\alpha |u_K(x)-u_K(x)\varphi_\alpha(x)| &\leq&   \lim_\alpha \norm{u_K-u_K\varphi_\alpha}_{\infty} \\
&\leq&   \lim_\alpha \norm{u_K-u_K\varphi_\alpha}_{A(H)} =0.
\end{eqnarray*}
So $\varphi_\alpha \rightarrow 1$ uniformly on compact subsets of $H$. Consequently,  by \cite[Lemma~4.4]{sk}, the existence of the net $(\varphi_\alpha)_\alpha$ implies $(P_2)$.

\vskip1.0em

$(P_2)\Rightarrow (B_1)$. \\
Let $(g_\beta)_\beta$ be the net generated by $(P_2)$ in \cite[Lemma~4.4]{sk}, that is $g_\beta=f_\beta*\tilde{f}_\beta$ for some $f_\beta\in L^2(H)$ while $\norm{f_\beta}_2=1$ for every $\beta$ and $g_\beta\rightarrow 1$ uniformly on compact sets. Therefore,
\[
1=\norm{f_\beta}_2^2= g_\beta(e) \leq \norm{g_\beta}_{\infty} \leq \norm{g_\beta}_{A(H)} \leq \norm{f_\beta}_2^2 \leq 1.
\]
 Also for each $u\in A(H) \cap C_c(H)$ and $f\in L^1(H)$,
 \begin{eqnarray*}
 \lim_\beta |\langle u g_\beta - u,f\rangle| &\leq& \lim_\beta \int_H |u(x)| | g_\beta(x)-1| |f(x)| dx\\
 &=& \int_{\supp(u)} |u(x)| |g_\beta(x) -1| |f(x)| dx=0. 
 \end{eqnarray*}

Let us fix $u\in A(H)$. For given  $\epsilon>0$ and $f\in L^1(H)$, there is some $v\in A(H) \cap C_c(H)$ such that $\norm{u-v}_{A(H)}<\epsilon$ and $\beta_0$ such that for any $\beta \succcurlyeq \beta_0$, $|\langle v g_\beta - v, f\rangle| <\epsilon$. So for any $\beta \succcurlyeq \beta_0$,
\begin{eqnarray*}
|\langle u g_\beta - u, f\rangle| &\leq& |\langle u g_\beta - v g_\beta, f\rangle| +|\langle v g_\beta - v , f\rangle| +|\langle v  - u , f\rangle| \\
&\leq & \norm{u - v}_{A(H)} \norm{ g_\beta}_{A(H)} \norm{f}_1  +\epsilon +\norm{v  - u}_{A(H)} \norm{ f}_1\\
& <&  \epsilon (2\norm{f}_1+1).
\end{eqnarray*}

Therefore, by one generalization to arbitrary functions on $A(H)$, $\lim_\beta ug_\beta = u$ in the topology $\sigma(A(H), L^1(H))$. But indeed $A(H)\subseteq B_\lambda(H)$ and this topology on bounded subsets of $A(H)$ coincides to the weak topology on $B_\lambda(H)$ i.e. $\sigma(B_\lambda(H), C^*_\lambda(H))$. So similar to the previous part, there is a $(e_\alpha)_\alpha \subset \operatorname{conv}\{g_\beta\}_\beta$ such that 
\[
\lim_\alpha \norm{ue_\alpha - e_\alpha}_{A(H)}=0
\]
for every $u\in A(H)$. Also note that for each $\alpha$, 
\[
1=e_\alpha(e) \leq \norm{e_\alpha}_{\infty} \leq \norm{e_\alpha}_{A(H)}\leq 1.
\]
\vskip1.0em

$(B_1)\Rightarrow (B_D)$ is trivial.
\end{proof}
}

\begin{eg}
Most of the discrete  regular Fourier hypergroup examples in \cite{mu1} including  \ma{Jacobi polynomial hypergroups}, \ma{cosh hypergroup}, and \ma{generalized Chebyshev polynomials} are hypergroups for them the support of the Plancherel measure includes the trivial character; hence, they satisfy $(P_2)$. Thus the Fourier algebra has a $1$-bounded approximate identity.
\end{eg}

Let us summarize the relation of amenability notions of a discrete regular Fourier hypergroup as follows.
\[\xymatrix{ 
{(F)} \ar@{=>}[r] &{(SF)} \ar@{=>}[r] & {(L_1)} \ar@{=>}[d] \ar@{=>}[rr] &                                  & (P_2) 
 \ar@{=>}[r]_{\nLeftarrow} & (P_1) 
\\
              &              &  {(L_D)} \ar@{=>}[r]                & {(B_D)} \ar@{<=>}[r]  & (B_1)\ar@{<=>}[u] 
                              &  
                            }
\]
 Recall that the implication $(L_1)\Rightarrow (P_2)$ is due to \cite[Proposition~4.4.3]{singh-mem}, as mentioned before, and $(P_2)\Rightarrow (P_1)$ is due to \cite[Theorem~4.3]{sk}. 
Also note that for a locally compact group $H$ all these conditions are equivalent and equal the amenability of the group.

\end{section}

\begin{section}{Amenability of hypergroup algebras}\label{s:AM-L1(H)}

In this section,    $H$ is   a discrete commutative hypergroup with the normalized Haar measure $\lambda$ (i.e. $\lambda(e)=1$).  In the following we rule out amenability of the hypergroup algebra of $H$ when the function defined by $x \mapsto \lambda(x)^{-1}$ belongs to $c_0(H)$.

 Since,   for every $x\in H$, $\lambda(x)\geq 1$,  $L^1(H)$ is  a subset of $\ell^1(H)$ and  $\norm{f}_{\ell^1(H)} \leq \norm{f}_{L^1(H)}$ for every $f \in L^1(H)$.  On the other hand, they are isometrically Banach algebra isomorphic through the  mapping 
\[
\iota:L^1(H) \rightarrow \ell^1(H), \ \ \ \ \delta_x \mapsto \lambda(x)\delta_x.
\]
The Fourier transform ${\cal F}:L^1(H) \rightarrow C(\wH)$ and the Fourier Stieltjes transform ${\cal FS}:\ell^1(H)\rightarrow C(\wH)$ are defined by
 \[
 \mathcal{F}(f)(\alpha) :=\sum_{x\in H} f(x) \overline{\alpha(x)}\lambda(x), \ \ \ \mathcal{FS}(\mu)(\alpha):= \sum_{x\in H} \mu(x) \overline{\alpha(x)}
 \]
 respectively,  for $f \in L^1(H)$, $\mu \in \ell^1(H)$, and $\alpha \in \widehat{H}$.  Hence, ${\cal FS}(\delta_e)(\alpha) =1$ for every $\alpha \in \widehat{H}$. 
 It is easy to check that    ${\cal FS}(\iota(f))={\cal F}(f)$ for every $f\in L^1(H)$.
We use ${\cal F}_2$ and ${\cal FS}_2$ to denote respectively the  Fourier and Fourier-Stieltjes transforms  on $H\times H$. 
Let us recall that, $L^1(H\times H)$ is isometrically Banach algebra isomorphic to  the projective tensor product, $L^1(H)\otimes_\gamma L^1( H)$.  We also    denote the Haar measure of $H \times H$ by $\lambda^2$, that is $\lambda^2(x,y)= \lambda(x) \lambda(y)$.

\vskip1.0em

 The  idea behind the proof of the following theorem  is from a study on weighted group algebras, \cite{bade}. Due to the similarity of the Haar measure of discrete hypergroups and weights on discrete groups, Lasser in \cite[Theorem~3]{la2} applied a similar argument to prove the following result for polynomial hypergroups. Here we  prove Lasser's result for discrete commutative hypergroups satisfying $(P_2)$.

\begin{theorem}\label{t:non-amenability-of-hypergroup-algebra}
Let $H$ be a discrete commutative  hypergroup which satisfies $(P_2)$. If $L^1(H)$ is amenable then   there is some $M\geq 1$ such that $\{x\in H: \lambda(x)\leq M\}$ is infinite.
\end{theorem}

\begin{proof}  Note that  $c_0(H \times H)$ forms an $L^1(H)$-bimodule by actions  $f\cdot  \phi := (\iota(f)\otimes \delta_{e})  *  \phi$ and $\phi \cdot f := (\iota(f) \otimes \delta_{e}) * \phi$ for every $f\in L^1(H)$ and $\phi \in c_0(H\times H)$ where $e$ is the trivial element of $H$.
Be aware of the difference between $*$, the convolution of $\ell^1(H\times H)$ and $*_{\lambda^2}$, the convolution on   $L^1(H\times H)$. 
By \cite[Proposition~1.2.36]{bl} and because $\iota$ is an isometric Banach algebra isomorphism, $f \cdot \phi$ and $\phi \cdot f$ also belong to $c_0(H \times H)$ and therefore, by these two actions  $c_0(H\times H)$ forms an $L^1(H)$-bimodule. 

Note that $L^1(H\times H)$ can be considered as the dual of $c_0(H\times H)$ with the dual product
\[
\langle \phi, \Phi \rangle :=\sum_{x,y \in H} \phi(x,y) \Phi(x,y) \lambda^2(x,y). \ \ \ \ (\phi \in c_0(H\times H), \ \Phi \in L^1(H\times H))
\]
Through this duality, $L^1(H\times H)$ becomes a dual  $L^1(H)$-bimodule where,
by \cite[Proposition~1.4.7]{bl},
\[
\langle f \cdot \Phi, \phi\rangle := \langle \Phi, f \cdot \phi\rangle = \langle \Phi, (\iota(f) \otimes \delta_e) * \phi\rangle = \langle (\iota(\tilde{f})\otimes \delta_e) * \Phi, \phi \rangle
\]
and similarly, $\Phi \cdot f = (\delta_e \otimes \iota(\tilde{f})) * \Phi$ for all $\Phi \in L^1(H \times H)$, $\phi \in c_0(H\times H)$, and $f \in L^1(H)$ where $\tilde{f}(x)= f(\tilde{x})$.

\vskip0.5em

Assume that $\lambda(x) \rightarrow \infty$.  Let $\phi_0\in c_0(H\times H)$ be defined by $\phi_0(x,y)=\lambda(x)^{-1} \lambda(y)^{-1}$.  Therefore, $\ker(\phi_0) \subseteq L^1(H \times H)$ is the weak$^*$ closed subspace of $L^1(H\times H)$ consisting of all $\Phi$ such that $\sum_{x,y} \Phi(x,y)=0$.  In particular $\Phi \in \ker(\phi_0)$ if and only if  ${\cal FS}_2(\Phi)(1,1)= \sum_{x,y} \Phi(x,y) =0$.
We  claim  that $\ker(\phi_0)$ is a dual  $L^1(H)$-bimodule.

 To prove this claim, let $f\in L^1(H)$ and $\Phi \in \ker(\phi_0)$. Therefore,
 \[
\mathcal{FS}_2(f \cdot \Phi)= \mathcal{FS}_2((\iota(\tilde{f})\otimes \delta_e) * \Phi) = \mathcal{FS}( \iota(\tilde{f})) \mathcal{FS}_2(\Phi)=0.
 \]
 Thus, $f \cdot \Phi$ also belongs to $\ker(\phi_0)$. Similarly, $\Phi \cdot f \in \ker(\phi_0)$. Therefore, 
  by \cite[Proposition~1.3]{bade}, $\ker(\phi_0)$ actually  forms a dual $L^1(H)$-bimodule.

  Let us define a linear mapping   $D(f):=\iota( \tilde{f})\otimes \delta_{e} - \delta_{e}\otimes \iota( \tilde{f})$ for $f \in L^1(H)$. It is clear that ${\cal FS}_2(D(f))(1,1)=0$ for every $f\in L^1(H)$; hence, $D:L^1(H)\rightarrow \ker(\phi_0)$. Also it is straightforward to check that $D$ is actually a derivation on $L^1(H)$ into the dual $L^1(H)$-bimodule $\ker(\phi_0)$. 
 
  \ignore{
   $D(f)=  {f}\cdot \delta_{e,e} - \delta_{e,e}  \cdot  {f}$ for every $f\in L^1(H)$. Therefore,  $D$ is a derivation. (Note that based on its definition, $D$ is not necessarily an inner derivation as $\delta_{(e,e)}$ does not belong to $\ker(\phi_0)$.) 
}

\vskip0.5em

Toward a contradiction assume that  $L^1(H)$ is amenable. Therefore, $D$ has to be inner that is,  there is some $\Phi_0 \in \ker(\phi_0)$ such that $D(f)=f\cdot \Phi_0 - \Phi_0 \cdot f$ for every $f \in L^1(H)$. 
For each pair $\alpha, \beta \in \widehat{H}$ with $\alpha \neq \beta$,    there exists  $f(=f_{(\alpha, \beta)}) \in L^1(H)$ such that  ${\cal F}(f )(\alpha)\neq {\cal F}(f )(\beta)$. Therefore,  
\begin{eqnarray*}
0\neq {\cal F}(f )(\alpha)-{\cal F}(f )(\beta)  &=&  {\cal FS}(\iota(f))(\alpha) - {\cal FS}(\iota(f))(\beta) \\
&=&  {\cal FS}_2(\iota(f)\otimes \delta_e  - \delta_e \otimes \iota(f))(\alpha, \beta) \\
&=&   {\cal FS}_2(D(\tilde{f})){(\alpha, \beta)}  \\
&=&  {\cal FS}(\iota(f))(\alpha) {\cal FS}_2(\Phi_0){(\alpha, \beta)} -{\cal FS}(\iota(f))(\beta) {\cal FS}_2(\Phi_0){(\alpha, \beta)}   \\
&=& \left({\cal F}(f)(\alpha)-{\cal F}(f )(\beta)\right)  {\cal FS}_2(\Phi_0){(\alpha, \beta)} .
\end{eqnarray*}
Consequently, ${\cal FS}_2(\Phi)(\alpha,\beta)=1$ for all $\alpha\neq \beta$ while ${\cal FS}_2(\Phi)(1,1)=0$. The continuity of ${\cal FS}_2(\Phi)$ implies that $(1,1)\in \wH\times\wH$ is an isolated point.   

Note that since $H\times H$ satisfies $(P_2)$, $(1,1)$ belongs to the support of the Plancherel measure. 
Therefore,  by \cite[Theorem~2.11]{voit} it is not possible that  $(1,1)\in \wH\times\wH$ is an isolated point unless $H$ is finite.  This finises the proof.
\end{proof}

\begin{subsection}{$ZA(G)$ of compact groups}\label{ss:AM-ZA(G)}

It was showed in \cite{ma} that for every compact group $G$, $\Irr(G)$ is a commutative discrete regular Fourier hypergroup and the hypergroup algebra of $\Irr(G)$ is isometrically Banach algebra isomorphic to $ZA(G)$, the subalgebra of the Fourier algebra of $G$ which is constructed by group characters.   $ZA(G)$ is called the \ma{central Fourier algebra} of $G$.

Amenability notions of $ZA(G)$ were studied in \cite{ma12} where a complete characterization for weak amenability was found. It was also shown that if $G$ is virtually abelian, $ZA(G)$ is amenable. Based on a result by Moore, \cite{moo}, for a virtual abelian group $G$, dimensions of the irreducible unitary representations are uniformly bounded. Therefore, $\{\lambda(\pi):  \pi \in \Irr(G)\}$ is a finite set. In \cite{ma12}, it was conjectured that the converse should also hold. 
The following corollary supports this conjecture.

\begin{cor}\label{c:AM-of-ZA(G))}
Let $G$ be a  non-discrete compact group such that  $\{\pi\in \wG:  d_\pi = n\}$ is finite for each positive integer $n$. Then $ZA(G)$ is not amenable.
\end{cor} 

\begin{proof}
It is known that $\Irr(G)$ as a strong hypergroup satisfies $(P_2)$, (look at \cite[Proposition~2.4]{ma14}).
Since $ZA(G)$  is isometrically Banach algebra  isomorphic to the hypergroup algebra of $\Irr(G)$, one  can apply Theorem~\ref{t:non-amenability-of-hypergroup-algebra} to  the hypergroup $\wG$ to finish the proof.
\end{proof}

One may compare Corollary~\ref{c:AM-of-ZA(G))} with   \cite[Theorem~6.1]{jo1} where Johnson proved a similar result for amenability of $A(G)$. 
 
\begin{rem}
The condition of Corollary~\ref{c:AM-of-ZA(G))} is far from being necessary for the amenability of $ZA(G)$. For example let $G=\Bbb{T}\times \operatorname{SU}(2)$.  By \cite[Proposition~2.1]{ma12}, $ZA(\operatorname{SU}(2))$ is not  weakly amenable  because it has a non-zero bounded point derivation, say $D_\theta$. Therefore, $D_\theta\otimes \varepsilon_e$ forms a symmetric non-zero bounded derivation on $ZA(\operatorname{SU}(2)) \otimes_\gamma A(\Bbb{T})$ when $\varepsilon_e(g):=g(e)$ for $e$ the identity of the group $\Bbb{T}$; hence, $ZA(\operatorname{SU}(2)\times \Bbb{T})\cong ZA(\operatorname{SU}(2)) \otimes_\gamma A(\Bbb{T})$ is not (weakly) amenable. On the other hand,  for each $n$ there are infinitely many  $\pi \in \wG$ such that $d_\pi =n$. 
\end{rem}

Let $G$ be  a compact group such that $d_\pi \rightarrow \infty$. Then $G$ is called a \ma{tall group}. Some properties of tall groups, specially profinite tall groups, have been studied in \cite{tall1,tall2,tall3}.

\begin{eg}\label{eg:amenability-of-Lie-groups}
Let us use the notations and facts mentioned in  the proof of Proposition~\ref{p:Leptin-number-of-SU(3)}. 
So based on Weyl's formula, one can easily show that 
$\tau(\pi) \leq |B|d_\pi$.
Hence,  $\{d_\pi\}_{\pi \in \Irr(G)}$ cannot be bounded when $B$ is finite. Therefore, $ZA(G)$ is not amenable. This class of compact Lie groups includes $\operatorname{SU}(n)$'s  for $n\geq 2$.
\end{eg}

\end{subsection}

\begin{subsection}{$Z\ell^1(G)$ of FC groups}\label{ss:amenability-of-zl1(G)}

 Let $G$ be a discrete \ma{finite conjugacy} or \ma{FC} group that is, for each conjugacy class $C$, $|C|<\infty$. Let us use $\conj(G)$ to denote the set of all conjugacy classes of $G$. $Z\ell^1(G)$ is  the center of the group algebra of $G$.  

It is known that $\conj(G)$ forms a discrete commutative hypergroup, \cite{bl}. It is straightforwrd to check that for each $C\in \conj(G)$, $\lambda(C)=|C|$ for $\lambda$ the normalized Haar measure of $\conj(G)$. It is known that  $Z\ell^1(G)$ is isometrically isomorphic to the  hypergroup algebra of $\conj(G)$. Also, since $\conj(G)$ is a strong hypergroup, $\conj(G)$ satisfies $(P_2)$ (look at \cite{hart}).

  Now, Theorem~\ref{t:non-amenability-of-hypergroup-algebra} immediately implies the following corollary.

\begin{cor}\label{c:ZL-AM-FC-groups}
Let $G$ be an infinite  FC group such that for every integer $n$ there are  only finitely many conjugacy classes $C$ such that  $|C|=n$. Then $Z\ell^1(G)$ is not amenable.
\end{cor}

\begin{proof}
As we saw before, $\conj(G)$ is a  discrete commutative hypergroup satisfying $(P_2)$.  Now one applies Theorem~\ref{t:non-amenability-of-hypergroup-algebra} and isomorphism $L^1(\conj(G))\cong \Zlg$ to finish the proof.
\end{proof}

For a group $G$, let $G'$ denote the \ma{derived subgroup} of $G$. It is immediate that if $G'$ is finite, for every $C \in \conj(G)$, $|C|\leq |G'|$.  The converse is also true i.e. if $\sup_{C\in \conj(G)} |C|<\infty$, then $|G'|<\infty$, see \cite{robinson}.

\begin{rem}
Recall that in \cite{AzSaSp}, it was proved  that for a locally compact group $G$ with a finite derived subgroup $G'$, $ZL^1(G)$ is amenable. 
Also for a specific class of FC groups, called RDPF groups, the amenability of $\Zlg$ is characterized in \cite{ma2}. The main result of \cite{ma2}  for a RDPF group $G$ is,  $\Zlg$ is amenable if and only if $|G'|<\infty$. These two studies suggest that the later characterization for RDPF groups may be  extendible  to general FC groups.
\end{rem}

\end{subsection}

\end{section}




\vskip2.0em


\section*{Acknowledgements}
This research was supported by a Ph.D. Dean's Scholarship at  University of Saskatchewan and a Postdoctoral Fellowship form the Fields Institute  and University of Waterloo.  These supports are gratefully acknowledged. The author also would like to express his deep gratitude to Yemon Choi, Ajit Iqbal Singh, Ebrahim Samei, and Nico Spronk for several constructive discussions and suggestions which improved the paper significantly. 
The author also thanks  Nico Spronk for directing him to \cite{zw-the}, Ajit Iqbal Singh for directing him to \cite{singh-mem}, and Jason Crann for directing him to \cite{izu}.
\footnotesize

\def\cprime{$'$} \def\cprime{$'$}

\vskip2.5em

{ 
\noindent {\normalsize Mahmood Alaghmandan}\newline\indent

\vskip1.0em

Department of Mathematical Sciences,\\
 Chalmers University of Technology and  University of Gothenburg,\\
  Gothenburg SE-412 96, Sweden

\vskip1em

Email: \texttt{mahala@chalmers.se}

\end{document}